\documentclass[12pt, a4paper, leqno]{amsart}

\usepackage[a4paper,left=25mm,right=25mm,
top=25mm,bottom=25mm,marginpar=22mm]{geometry}

\usepackage{amsmath}
\usepackage{amsfonts}
\usepackage{amssymb}
\usepackage{color}
\usepackage{comment,graphicx,esint}
\usepackage[T1]{fontenc} 
\usepackage{url}
\usepackage[colorlinks]{hyperref}
\usepackage{ae} 

\renewcommand{\epsilon}{\varepsilon}
\renewcommand{\phi}{\varphi}
\newcommand{\epsi}{\varepsilon}

\numberwithin{equation}{section}

\newtheoremstyle{thmlemcorr}{10pt}{10pt}{\itshape}{}{\bfseries}{.}{10pt}{{\thmname{#1}\thmnumber{
#2}\thmnote{ (#3)}}}
\newtheoremstyle{thmlemcorr*}{10pt}{10pt}{\itshape}{}{\bfseries}{.}\newline{{\thmname{#1}\thmnumber{
#2}\thmnote{ (#3)}}}
\newtheoremstyle{defi}{10pt}{10pt}{}{}{\bfseries}{.}{10pt}{{\thmname{#1}\thmnumber{
#2}\thmnote{ (#3)}}}
\newtheoremstyle{remexample}{10pt}{10pt}{}{}{\bfseries}{.}{10pt}{{\thmname{#1}\thmnumber{
#2}\thmnote{ (#3)}}}
\newtheoremstyle{ass}{10pt}{10pt}{}{}{\bfseries}{.}{10pt}{{\thmname{#1}\thmnumber{
A#2}\thmnote{ (#3)}}}

\theoremstyle{thmlemcorr}
\newtheorem{theorem}{Theorem}
\numberwithin{theorem}{section}
\newtheorem{lemma}[theorem]{Lemma}

\newtheorem{proposition}[theorem]{Proposition}

\theoremstyle{thmlemcorr*}
\newtheorem{theorem*}{Theorem}
\newtheorem{lemma*}[theorem]{Lemma}
\newtheorem{corollary*}[theorem]{Corollary}
\newtheorem{proposition*}[theorem]{Proposition}
\newtheorem{problem*}[theorem]{Problem}
\newtheorem{conjecture*}[theorem]{Conjecture}

\theoremstyle{defi}
\newtheorem{definition}[theorem]{Definition}

\theoremstyle{remexample}
\newtheorem{remark}[theorem]{Remark}

\providecommand{\loc}{{\ensuremath{\mathrm{loc}}}}

\newcommand{\R}{\mathbb{R}}
\newcommand{\N}{\mathbb{N}}

\newcommand{\Rn}{{\mathbb{R}^n}}

\newcommand{\RR}{\mathbb{R}}
\newcommand{\NN}{\mathbb{N}}

\renewcommand{\le}{\leqslant}
\renewcommand{\ge}{\geqslant}
\renewcommand{\leq}{\leqslant}
\renewcommand{\geq}{\geqslant}

\newcommand{\px}{{p(\cdot)}}
\newcommand{\qx}{{q(\cdot)}}
\newcommand{\phix}{{\phi(\cdot)}}
\newcommand{\psix}{{\psi(\cdot)}}

\newcommand{\hatrho}{{\psi_\epsilon}}
\newcommand{\phiu}{{\phi_\uparrow}}
\newcommand{\phid}{{\phi_\downarrow}}
\newcommand{\cu}{{c_\uparrow}}
\newcommand{\cd}{{c_\downarrow}}

\DeclareMathOperator{\dist}{dist}

\usepackage{ifthen}
\usepackage{enumitem}

\providecommand{\varitem}{} 
\makeatletter
\newenvironment{hypotheses}[1]
 {\renewcommand\varitem[1]{\item[
 {#1\arabic{enumi}\rlap{##1}.}]%
    \edef\@currentlabel{#1\arabic{enumi}{##1}}}%
  \enumerate[label={({#1\arabic*})}, ref=#1\arabic*,
  leftmargin=18mm]}
 {\endenumerate}
\makeatother

\providecommand{\varitem}{} 
\makeatletter
\newenvironment{hypothesesD}[1]
 {\renewcommand\varitem[1]{\item[
 {#1\rlap{##1}.}]%
    \edef\@currentlabel{#1{##1}}}%
  \enumerate[label={({#1})}, ref=#1]}
 {\endenumerate}
\makeatother

\providecommand{\varitem}{} 


\newcommand{\aev}{{a.e.}}

\begin{document}

\title{Characterization of generalized Orlicz spaces}

\author[Rita Ferreira]{Rita Ferreira}
\address[R.~Ferreira]{King Abdullah University
of Science
and Technology (KAUST), CEMSE Division,  Thuwal 23955-6900,
Saudi Arabia.}
\email{rita.ferreira@kaust.edu.sa}
\urladdr{http://www.ritaferreira.pt}

\author[Peter H\"{a}st\"{o}]{Peter H\"{a}st\"{o}}
\address[P.~H\"{a}st\"{o}]{Department of Mathematical Sciences,
P.O.\ Box 3000, FI-90014 University of Oulu, Finland
\\ and
\\
Department of Mathematics and Statistics, University of Turku, Finland}
\email{peter.hasto@oulu.fi}
\urladdr{http://cc.oulu.fi/$\sim$phasto/}


\author[Ana Margarida Ribeiro]{Ana Margarida Ribeiro}
\address[A.M.~Ribeiro]{Centro de Matem\'{a}tica e Aplica\c{c}\~{o}es (CMA) and Departamento de Matem\'{a}tica, Faculdade de Ci\^{e}ncias e Tecnologia, Universidade Nova de Lisboa,
Quinta da Torre, 2829-516 Caparica, Portugal}
\email{amfr@fct.unl.pt}
\urladdr{https://sites.google.com/site/anaribeirowebpage/home}


\begin{abstract}
The norm in classical Sobolev spaces can be expressed as a difference
quotient. This expression can be used to generalize the space
to the fractional smoothness case. Because the difference
quotient is based on shifting the function, it cannot be
used in generalized Orlicz spaces. In its place,
we introduce a smoothed difference quotient and show that
it can be used to characterize the generalized
Orlicz--Sobolev space. Our results are new even in
Orlicz spaces and variable exponent spaces.
\end{abstract}

\date{\today}

\maketitle



\section{Introduction}

Bourgain, Brézis, and Mironescu \cite{BoBrMi1,BoBrMi2} studied the limit
behavior of the Gagliardo semi-norms
\[
||f||^p_{W^{s,p}} =
\int_\Omega \int_\Omega \frac{|f(x)-f(y)|^p}{|x-y|^{n+sp}} \, dx\, dy,
\quad 0<s<1,
\]
as $s\to 1$, and established the appropriate scaling factor for comparing the limit
with the $L^p$-norm of the gradient of $f$. They characterized the Sobolev space $W^{1,p}$ and
proved the convergence of certain imaging models of Aubert and Kornprobst \cite{AK09} to
the well-known total variation model of Rudin, Osher, and Fatemi \cite{ROF92}.
Our aim in this paper is to extend the characterization to generalized Orlicz
spaces defined on open subsets of $\Rn$.

Generalized Orlicz spaces $L^\phix$ have been studied since the 1940's.
A major synthesis of functional analysis in these spaces is given in the
monograph of Musielak \cite{Mus83} from 1983, for which
reason they have also been called Musielak--Orlicz spaces. These
spaces are similar to
Orlicz spaces, but defined by a more general function $\phi(x,t)$
that may vary with the location in space:
the norm is defined by means of the integral
\[
\int_\Rn \phi(x, |f(x)|)\, dx,
\]
whereas in an Orlicz space, $\phi$ would be independent of $x$,
$\phi(|f(x)|)$. When $\phi(t)=t^p$, we obtain the
Lebesgue spaces, $L^p$.
Generalized Orlicz spaces are motivated
by applications to image processing \cite{CheLR06, HarHLT13},
fluid dynamics \cite{Swi14}, and differential equations
\cite{BarCM16, GiaP13}. Recently, harmonic analysis in this
setting has been studied e.g.\ in \cite{CruH_pp16, Has15, MaeMOS13a}.

We have in mind two principal classes of examples of
generalized Orlicz spaces:
variable exponent spaces $L^\px$, where
$\phi(x,t):= t^{p(x)}$ \cite{CruF13, DieHHR11},
and
dual phase spaces, where $\phi(x,t):= t^p + a(x) t^q$
\cite{BarCM15, BarCM16, ColM15a, ColM15b, ColM16}.
It is interesting to note that our general methods give optimal
results in these two disparate cases, cf.\ \cite{Has15}.
Also covered are variants of the variable exponent
case such as $t^{p(x)}\log(e+t)$ \cite{GiaP13, MizOS08, Ok16a, Ok16b}.

It is not difficult to see that a direct generalization of the difference quotient
to the non-translation invariant generalized Orlicz case is not possible
\cite[Section~1]{HasR_pp15}. For instance Besov and Triebel--Lizorkin spaces
in this context have been defined using Fourier theoretic approach
\cite{AlmH10, DieHR09}. Hästö and Ribeiro \cite{HasR_pp15},
following \cite{DieH07}, adopted a more
direct approach with a smoothed difference quotient expressed by means of the
sharp averaging operator $M^\#_{B(x,r)}$. This is the general approach adopted also
in this paper. This paper improves \cite{HasR_pp15} in three major
ways:
\begin{enumerate}
\item
Instead of variable exponent spaces, we consider more general generalized Orlicz
spaces;
\item
Instead of $\Rn$, we allow arbitrary open sets $\Omega\subset\Rn$; and
\item
In our main result, we relax the technical assumption
$f\in L^1(\Omega)$ to its natural form; i.e.,\ $f\in L^1_\loc(\Omega)$.
\end{enumerate}
The latter two generalizations have been previously established in
the case $L^p$ by Leoni and Spector \cite{LeoS11}, see also \cite[Section~1]{FerKR15}.
In order to achieve these goals, the methods of the main results
(Section~\ref{sect:main}) are completely different from
those in \cite{HasR_pp15} and involve a new bootstrapping scheme.

\bigskip

We introduce some notation to state our main result. We refer to the
next section for the precise definition of $L^\phix$ and the assumptions
in the theorem. Let \(\Omega\subset \RR^n\) be an open set,
\(\phi \in \Phi_w(\Omega) \),  and \(\epsi>0\). Let \(\psi_\epsi\)
be a set of functions such that
\begin{equation}\label{condition rho 1}
\hatrho \in L^1(0,\infty), \quad \hatrho \ge 0,\quad
\int_0^\infty \hatrho(r)\, dr = 1,
\end{equation}
and, for every $\gamma>0$,
\begin{equation}\label{condition rho 2}
\lim_{\epsilon\to 0^+} \int_\gamma^\infty \hatrho(r)\,
dr =0.
\end{equation}
We define a weak quasi-semimodular, \(\varrho^\epsilon_{\#,\Omega}(f)\),
on \(L^1_{\rm loc}(\Omega)\) by setting
\begin{align*}
\varrho^\epsilon_{\#,\Omega}(f):=
\int_0^\infty \int_{\Omega_r} \phi \left(x, \tfrac1r M^\#_{B(x,r)}f\right) \, dx \,  \hatrho(r)\, dr\quad
\end{align*}
for  \(f\in L^1_{\rm loc}(\Omega)\), where $\Omega_r:=\{x\in\Omega:\ \dist(x,\partial\Omega)>r\}$
and\[
M^\#_{B(x,r)}f := \fint_{B(x,r)} | f(y) - f_{B(x,r)} | \, dy\quad \text{with}\quad
f_{B(x,r)}:=\fint_{B(x,r)}f(y)\,dy.
\]
The associated quasi-norm, \(\| f\|^\epsilon_{\#,\Omega}  \), is defined by
\[
\| f\|^\epsilon_{\#,\Omega}  := \inf\big\{ \lambda>0 \,|\, \varrho^\epsilon_{\#,\Omega}(f/\lambda) \le 1\big\}.
\]

The following is our main result stated for a $\Phi$-function---it is also possible
to state it for weak $\Phi$-functions, see Theorem~\ref{main theorem gen}.
The proof follows from Propositions~\ref{prop:SobolevCase} and
\ref{prop:showsSobolev}.

\begin{theorem}\label{main theorem}
Let $\Omega\subset\Rn$ be an open set, let $\phi\in \Phi(\Omega)$ satisfy 
Assumptions~\eqref{A}, \eqref{aInc}, and \eqref{aDec},
and let $(\hatrho)_\epsilon$ be a family of functions
satisfying \eqref{condition rho 1} and \eqref{condition rho 2}.
Assume  that $f\in L^1_{\rm loc}(\Omega)$.
Then,
\begin{equation*}
\begin{aligned}
\nabla f\in L^{\phix}(\Omega;\Rn)
\quad\Leftrightarrow\quad
\limsup_{\epsilon\to 0^+} \varrho^\epsilon_{\#,\Omega}(f) <\infty.
\end{aligned}
\end{equation*}
In this case,
\begin{equation*}
\lim_{\epsilon\to0^+} \varrho^\epsilon_{\#,\Omega}(f)
= \varrho_{\phix,\Omega}(c_n |\nabla f|)
\qquad\text{and}\qquad
\lim_{\epsilon\to0^+}\| f\|^\epsilon_{\#,\Omega} = c_n \| \nabla f\|_{\phix,\Omega},
\end{equation*}
where $c_n := \fint_{B(0,1)} |x\cdot e_1| \, dx$.
\end{theorem}

\begin{remark}
Note that the previous result is new even in the case of classical Orlicz spaces.
In this case, Assumption~\eqref{A} automatically holds.
\end{remark}

\begin{remark}
In the case $\phi(x,t)=t^{p(x)}$,
Assumptions~\eqref{A0} and \eqref{loc} always hold,
while Assumptions~\eqref{A1} and \eqref{A2} are equivalent to
the local $\log$-Hölder continuity and Nekvinda's decay condition, respectively.
Moreover, if \(p_-:=\inf_{x\in\Omega}
p(x) >1\), then \eqref{aInc} holds;
and if  \(p_+:=\sup_{x\in\Omega}
p(x) <\infty\), then \eqref{aDec} holds.
\end{remark}


\section{Preliminaries}\label{sec:prelim}

This section is organized as follows. In Subsection~\ref{subsec:not},  we collect some notation
used throughout this paper. Then, in Subsection~\ref{subsec:PhiOrlicz},
we recall the definition of  \(\Phi\)-functions and of
some of its  generalizations; we  recall also the associated
Orlicz spaces, norms, and semimodulars.
Finally, in Subsection~\ref{subsec:mainass}, we introduce
and discuss our
main assumptions on the (generalized weak) \(\Phi\)-functions and  relate
them with other assumptions
in the literature. We conclude by proving two auxiliary
results, Lemmas~\ref{lem:wpm} and \ref{lem:key-estimate}. The first one can be interpreted as  a counterpart in our setting of the weighted power-mean inequality for the function \(\varphi(x, t)
= t^p\) for some \(p\geq 1\); the second one is a Jensen-type inequality
in the spirit of
\cite[Lemma~4.4]{Has15}
and \cite[Lemma~2.2]{HasR_pp15}.

\subsection{\sc Notation}\label{subsec:not}

We denote by $\mathbb{R}^{n}$ the $n$-dimensional real
Euclidean space. We write $B(x,r)$ for the open ball
in
$\mathbb{R}^{n}$ centered at $x\in \mathbb{R}^{n}$ and
with
radius $r>0$.
We use $c$ as a generic positive constant; i.e.,\ a constant
whose
value may change from appearance to appearance.
If $E\subset {\mathbb{R}^{n}}$ is a  measurable
set, then $|E|$ stands for its (Lebesgue) measure and
$\chi_{E}$ denotes its characteristic function;
we denote by \(L^0(E)\)  the space of all Lebesgue measurable functions on \(E\).

 Let \(g,\, h: D\subset\mathbb{R}^m \to [0,\infty]\) be
two functions. We write \(g\lesssim h\) to mean that there exists a positive constant,
\(C\), such that \(g(z)\leq C h(z)\)
for all
\(z\in D\). If  \(g\lesssim h \lesssim g\), we write   \(g\approx h\).
Also, given a sequence \((g_\epsi)_\epsi\) of non-negative
functions in \(D\), the notation \(\lim_{\epsi\to 0} g_\epsi
\approx h\) means that there is a positive constant,
\(C\), such that \(\tfrac1{C} h\leq\liminf_{\epsi\to0} g_\epsi
\leq \limsup_{\epsi\to0} g_\epsi \leq C h\) in \(D\); i.e.,\
for the equivalence we do not require the limit to exist, only
the upper and lower limits to be within a constant of each other.

Moreover, if \(D=[0,\infty) \), we say that \(g\) and \(h\) are
\textit{equivalent}, written \(g\simeq h\), if there exists \(L\geq 1\)
such that for all \(t\ge 0\), we have  \(h(\frac{t}{L})
\leq g(t) \leq h(Lt)\).
In this case,  \(L\) is said to be the \textit{equivalence constant}.

Let \(\varphi:[0,\infty) \to [0,\infty]\) be an increasing
function.  We denote by \(\phi^{-1}:[0,\infty]\to[0,\infty]\) the
 left-continuous generalized inverse of \(\phi\); that
is, for all \(s\in[0,\infty]\),
\begin{equation}\label{definverse}
\begin{aligned}
\phi^{-1}(s) := \inf\{t\geq0|\,\phi(t)\geq s\}.
\end{aligned}
\end{equation}
Let \(\Omega\subset\RR^n\) be  open,  let \(\phi:\Omega\times
[0,\infty) \to [0,\infty]\), and
let \(B\subset\RR^n\). Then, \(\phi^+_B:[0,\infty) \to
[0,\infty]\) and \(\phi^-_B:[0,\infty) \to
[0,\infty]\) are the functions defined for all \(t\in
[0,\infty)\) by
\begin{equation*}
\begin{aligned}
\phi^+_B(t):= \sup_{x\in \Omega\cap B} \phi(x,t) \quad
\hbox{and} \quad \phi^-_B(t):= \inf_{x\in \Omega\cap B} \phi(x,t).
\end{aligned}
\end{equation*}
Note that if $\phi(x,\cdot)$ is increasing for  every \(x\in\Omega\), then so are $\phi^+_B(\cdot)$ and $\phi^-_B(\cdot)$;  thus, these functions admit a left-continuous  generalized inverse in the sense of \eqref{definverse}.


\subsection{\sc \texorpdfstring{$\Phi$}{Phi}-functions
and generalized Orlicz spaces}\label{subsec:PhiOrlicz}

We start this subsection by introducing the notion of
almost increasing and almost decreasing functions, after
which we recall the definition of  \(\Phi\)-functions and of
some of its  generalizations.

\begin{definition}
We say that a function \(g:D\subset\R \to [0,\infty]\)
is \emph{almost increasing} if \(g(t_1) \leq c\, g(t_2)\) for every
\(t_1\leq t_2\) in $D$ and some $c$.
We say that \(c\) is the \emph{monotonicity constant}
of \(g\). \emph{Almost decreasing} is defined analogously.
\end{definition}

\begin{definition}
Let  \(\phi:[0,\infty) \to [0,\infty]\)
be an increasing function satisfying
\(\phi (0)= \lim_{t\to0^+} \phi(t)=0\)
and \( \lim_{t\to\infty} \phi (t) = \infty\). We say that \(\phi\) is a
\begin{enumerate}
\item[(i)]
\emph{weak \(\Phi\)-function}
if  \(t\mapsto\frac{\phi(t)}{t} \) is almost increasing.
\item[(ii)]   \emph{\(\Phi\)-function}
if it is left-continuous and convex.
\end{enumerate}
We denote  by \(\Phi_w\)   the set of all weak \(\Phi\)-functions
and by \(\Phi\) the set of all  \(\Phi\)-functions.
\end{definition}

\begin{remark}\label{onPhis}
If $\phi\in \Phi$, then, by convexity and because \(\phi(0)=0\), for \(0<t_1<t_2\), $\phi(t_1)= \phi(\tfrac{t_1}{t_2}t_2+
(1-\tfrac{t_1}{t_2})0) \le \tfrac{t_1}{t_2}  \phi(t_2)$;
 thus, $t\mapsto \frac{\phi(t)}t$ is increasing. Hence, \(\Phi\subset \Phi_w\).
Conversely, if $\phi\in \Phi_w$, then there exists $\psi\in \Phi$ such
that $\phi\simeq \psi$ \cite[Proposition~2.3]{HarHK_pp16}.
\end{remark}

\begin{definition}
Let \(\Omega\subset \RR^n\) be an open set, and let \(\phi: \Omega \times
[0,\infty) \to [0,\infty]\) be a function such that   \(\phi(\cdot,t)\in
L^0(\Omega) \) for every \(t\in [0,\infty)\).
We say that \(\phi\) is a
\begin{itemize}
\item[(i)]
\emph{generalized weak \(\Phi\)-function}
on \(\Omega\) if \(\phi(x,\cdot)\in
\Phi_w\) uniformly in \(x\in\Omega\);
i.e.,\ the monotonicity constant is independent of \(x\).
\item[(ii)]
\emph{generalized  \(\Phi\)-function}
on \(\Omega\) if \(\phi(x,\cdot)\in
\Phi\) for  every \(x\in\Omega\).
\end{itemize}
We denote  by \(\Phi_w(\Omega)\) and \(\Phi(\Omega)\) the sets of
generalized weak \(\Phi\)-functions and generalized
 \(\Phi\)-functions, respectively.
\end{definition}

By this definition, it is clear that properties of (weak) $\Phi$-functions
carry over to generalized (weak) $\Phi$-functions point-wise uniformly. In particular,
this holds for Remark~\ref{onPhis}. Similarly, $\phi\simeq \psi$ means that
$\phi(x,\cdot)\simeq \psi(x,\cdot)$ with constant uniform in $x$, etc.

Next, we  recall the definition of the generalized Orlicz space, quasi-norm,
and quasi-semimodular associated with a generalized weak
 \(\Phi\)-function.

\begin{definition}\label{defOS}
Let \(\Omega\subset \RR^n\) be an open set,  let
\(\phi \in \Phi_w(\Omega)\), and   consider the \emph{weak quasi-semimodular},
\(\varrho_{\phix,\Omega}\), on \(L^0(\Omega)\) defined
by
\begin{equation*}
\varrho_{\phix,\Omega}(f):= \int_\Omega \phi (
x, |f(x)|) \,dx
\end{equation*}
for all \(f\in L^0(\Omega)\). The \emph{generalized Orlicz
space}, \(L^\phix(\Omega)\),  is given by
\begin{equation*}
L^\phix(\Omega):= \left\{f\in L^0(\Omega)|\, \varrho_{\phix,\Omega}(\lambda
f) < \infty \hbox{ for some } \lambda>0 \right\}.
\end{equation*}
We endow \(L^\phix(\Omega)\) with the quasi-norm
\begin{equation*}
\Vert f \Vert_{\phix,\Omega} := \inf \left\{ \lambda>0|
\, \varrho_{\phix,\Omega}\left(f/\lambda\right)
\leq 1 \right\}\!.
\end{equation*}
\end{definition}

If, in the Definition~\ref{defOS}, \(\phi\in \Phi(\Omega)\) , then \(\varrho_{\phix,\Omega}(\cdot)\)
defines a semimodular on \(L^0(\Omega)\)  and \(\Vert \cdot \Vert_{\phix,\Omega}\) a norm on \(L^{\phi(\cdot)}(\Omega)\)
(see \cite{DieHHR11}).

\begin{remark}\label{casePhi}
If \(\phi,\psi \in \Phi_w(\Omega)\) and  \(\phi\simeq\psi\),
then $L^\phix=L^\psix$ with equivalent quasi-norms.
\end{remark}


\subsection{\sc Main assumptions}\label{subsec:mainass}

We begin by introducing our main assumptions
on the generalized (weak) \(\Phi\)-functions on \(\Omega\).
The first three assumptions, \eqref{doubling}, \eqref{aInc}, and \eqref{aDec},  extend three known properties for $\Phi$-functions
 to generalized $\Phi$-functions point-wise
uniformly. The fourth assumption,
\eqref{A}, relates the behavior of generalized  \(\Phi\)-functions
at different
values of the variable in $\Omega$.
The last assumption, \eqref{loc}, implies that 
simple functions belong to the generalized Orlicz space.
These last two assumptions hold trivially for Orlicz spaces.

Let \(\Omega\subset \RR^n\) be open and \(\phi \in \Phi_w(\Omega)\). We denote by $\phi^{-1}$ the generalized inverse of $\phi$ with respect
to the
second variable (see \eqref{definverse}).
\begin{hypothesesD}{$\Delta_2$}
\item \label{doubling}
$\phi$ is \emph{doubling}; i.e.,\
there exists \(A>0\) such that \(\phi(x,2t) \leq A\phi(x,t)\)
for \aev\ \(x\in\Omega\) and for all \(t\geq 0\).
\end{hypothesesD}
\begin{hypothesesD}{aInc}
 \item \label{aInc}
There exists a constant \(\phiu>1\) such that for \aev\ \(x\in\Omega\),  the map \(s\mapsto
 s^{-\phiu}\phi(x,s) \) is almost increasing with monotonicity constant
\(\cu\) independent of \(x\).
\end{hypothesesD}
\begin{hypothesesD}{aDec}
\item \label{aDec}
There exists a constant \(\phid>1\) such that for \aev\ \(x\in\Omega\),  the map \(s\mapsto
 s^{-\phid}\phi(x,s) \) is almost decreasing with monotonicity constant
\(\cd\) independent of \(x\).
\end{hypothesesD}

\begin{hypothesesD}{A}
\item
\label{A}
There exist \(\beta,\sigma>0\) for which:
\end{hypothesesD}

\begin{hypotheses}{A}\setcounter{enumi}{-1}
\item \label{A0}
\(\phi(x,\beta \sigma) \le 1 \le \phi(x,\sigma)\) for all\ \(x\in\Omega\);
\item \label{A1}
\(\phi(x,\beta t) \le \phi(y,t)\)
for every ball \(B\subset \Omega\),  \(x,\, y\in B\), and
 \(t\in \big [\sigma, \phi^{-1}\big(y,\frac{1}{|B|}\big)
\big]\);
\item \label{A2}
there exists \(h\in
L^1(\Omega) \cap L^\infty(\Omega)\) such that
for  \aev\ \(x, \, y\in\Omega\) and for all \(t\in
[0,\sigma]\), we have \(\phi(x, \beta t) \le \phi(y,
t) + h(x) + h(y).\)
\end{hypotheses}

\begin{hypothesesD}{loc}
 \item \label{loc} There exists \(t_0>0\) such that
   \( \phi(\cdot,t_0) \in L^1_{\rm loc} (\Omega) \).
\end{hypothesesD}

The notation \eqref{aInc}$_1$ is used for a version of
\eqref{aInc} with \(\phiu\ge 1\); i.e.,\ equality included.
Note that, for any weak $\Phi$-function,
\eqref{aInc}$_1$  holds  for $\phiu=1$.

\begin{remark}\label{Remark:assumptions}
Let us collect several observations regarding the assumptions above.
\begin{enumerate}
\item
Each of the previous conditions is invariant under equivalence
of (weak) $\Phi$-functions; i.e.,\ if $\phi\simeq\psi$, then $\phi$ satisfies a
condition if and only if $\psi$ satisfies it.
\item
For doubling (weak) $\Phi$-functions, \(\simeq\) and \(\approx\) are equivalent.
\item
\eqref{aDec} and \eqref{doubling} are equivalent \cite[Lemma~2.6]{HarHT_pp16}.
\item
\eqref{A0} implies \eqref{loc} (choose $t_0:=\beta \sigma$).
\item
If \eqref{aInc} and \eqref{aDec} hold, then $\phi^\uparrow \le \phi^\downarrow$.
\item
Finally, note that if $\phi$  satisfies \eqref{aDec}
and \eqref{loc},
then \( \phi(\cdot,t) \in L^1_{\rm loc} (\Omega)\) for
every \(t\geq0\).
\end{enumerate}
\end{remark}

It follows directly from the definition of the left-inverse that
$\phi\simeq \psi$ implies $\phi^{-1}\approx \psi^{-1}$.
Furthermore, if $\phi\simeq \psi$, then
$\phi_B^-\simeq \psi_B^-$ and so  $(\phi_B^-)^{-1}\approx (\psi_B^-)^{-1}$.

The next proposition shows that  Assumption
\eqref{A1} is equivalent to its counterpart in
\cite{Has15}.

\begin{lemma}\label{onA1s}
Let \(\phi\in \Phi_w(\Omega)\), \(\sigma>0\), and \(\beta\in
(0,1)\). Then, \(\phi\) satisfies \eqref{A1}  for \((\sigma,\beta)\)
if and only if for every ball \(B\subset \Omega\), and for all finite
\(t\in [\sigma, (\phi^-_B)^{-1}(1/|B|)]\), we have
\begin{equation}
\label{A1MO}
\phi^+_B(\beta t) \leq \phi^-_B(t).
\end{equation}
\end{lemma}

\begin{proof}
Because \((\phi^-_B)^{-1}\big(\frac{1}{|B|}\big) \geq\phi^{-1}\big(y,\frac{1}{|B|}\big) \), it follows that if \(\phi\) satisfies \eqref{A1MO},
then it also satisfies \eqref{A1}.

Conversely, assume that \(\phi\) satisfies \eqref{A1}.
We first consider the case when \(t\in \big[\sigma, (\phi^-_B)^{-1}
\big(\frac{1}{|B|}\big)\big)\).
In this case, we can find \((y_i)_{i\in\NN}\subset  B\) such that  \(t\in \big[\sigma, \phi^{-1}\big(y_i,\frac{1}{|B|}\big)\big]\)
for all \(i\in\NN\) and
\(\phi^-_B(t) = \lim_{i\to\infty} \phi(y_i,t)\). Then,
by \eqref{A1},  we have
\begin{equation*}
\begin{aligned}
\phi(x,\beta t) \leq \phi(y_i,t)
\end{aligned}
\end{equation*}
for  \aev\ \(x\in B\) and for all \(i\in\NN\). Taking
the supremum over \(x\in B\) and then letting \(i\to
\infty\) in the previous estimate, we obtain \(\phi^+_B(\beta t) \leq \phi^-_B(t)\). Finally, assume that \(t= (\phi^-_B)^{-1}
\big(\frac{1}{|B|}\big)<\infty\), and let \(t'\in [\sigma, t)\).
By the previous case,
\begin{equation}\label{AiMO1}
\begin{aligned}
\phi(x,\beta t') \leq \phi^+_B(\beta
t') \leq \phi^-_B(t') \leq  \phi^-_B(t)
\end{aligned}
\end{equation}
for \aev\ \(x\in B\), where in the last inequality we used
the fact that \(\phi^-_B\) is increasing. Taking the
limit \(t'\to t\) in \eqref{AiMO1}, the left-continuity
of \(\phi(x,\cdot)\) yields  \(\phi(x,\beta t) \leq \phi^-_B(t)\)
for \aev\ \(x\in B\). Hence, taking
the supremum over \(x\in B\), we conclude that \(\phi
\) satisfies \eqref{A1MO}.
\end{proof}


The next lemma shows that left-inverse commutes with infimum, even
when the function is not continuous.

\begin{lemma}\label{Lem:auxformain}
Let \(\Omega\subset \RR^n\) be an open
set,   let \(B\subset \Omega\), and
let $\phi\in \Phi_w(\Omega)$.
Then, \((\phi^-_B)^{-1} = (\phi^{-1})_B^+\).
\end{lemma}

\begin{proof}
Fix \(s\in[0,\infty]\). For every \(x\in
B\), we have
\[
(\phi_B^-)^{-1}(s)
 =
\inf\big\{t\geq0\,|\, \phi_B^-(t)\ge s\big\}
 \geq
\inf\big\{t\geq0\,|\,  \phi(x,t)\ge s\big\}=\phi^{-1}(x,s).
\]
 Thus, taking the supremum over \(x\in
B\),
\[
(\phi_B^-)^{-1}(s)
 \geq
(\phi^{-1})_B^+(s) .
\]

To prove the converse inequality, we
may assume that \(\bar t: = (\phi^{-1})_B^+(s)
< \infty\) without loss of generality.
Fix \(\epsi>0\). By definition of \(\bar
t\) and because \(\phi(x,\cdot)\)
is increasing for every \(x\in B\), we
have \(\phi(x,\bar t + \epsi) \geq s\)
for all \(x\in B\). Hence, \((\phi_B^-)^{-1}(s)
\leq \bar t + \epsi\). Letting \(\epsi\to0\),
we conclude the desired inequality.
\end{proof}

The following lemma allows us to relate Assumption~\eqref{A1} with its counterpart in
\cite{HarH_pp15}.

\begin{lemma}\label{lem:As}
Let \(\Omega\subset \RR^n\) be  open
 and assume that $\phi \in \Phi(\Omega)$ is doubling and satisfies \eqref{A0}.
Then, \eqref{A1} is equivalent to the following condition: 
%
\begin{itemize}
\item[(A1')] \label{A1'}
$\phi^{-1}(x,s) \lesssim \phi^{-1}(y,s)$ for  every ball
\(B\subset\Omega\),   $x,\,y\in B$,
and  $s\in [1,\frac1{|B|}]$.
\end{itemize}
%
\end{lemma}

\begin{proof}
%
Because $\phi$ belongs to $\Phi(\Omega)$ and is doubling,
it is a bijection with respect to the second variable from $[0,\infty)$ to $[0,\infty)$.
Applying $\phi^{-1}$ to \eqref{A0} and \eqref{A1}, we find that
\begin{equation}\label{equiv A0}
\eqref{A0} \quad \Leftrightarrow \quad \beta\sigma \le \phi^{-1}(x,1) \le \sigma\
\text{for}\ \aev\ x\in\Omega;
\end{equation}
\begin{equation*}
\eqref{A1} \quad \Leftrightarrow \quad \beta \phi^{-1}(x,s) \le \phi^{-1}(y,s)\ \text{for}\ \aev\ x,\,y \in B\ \text{and for all}\ s\in \Big[\phi(x,\sigma),\tfrac1{|B|}\Big].
\end{equation*}
If $s\in [1,\phi(x,\sigma)]$, then \(\beta\phi^{
-1}(x,1)
\leq \beta\phi^{-1}(x,s)\leq\beta \sigma\) because \(\phi^{-1}(x,\cdot)\)
is increasing. Thus, using \eqref{equiv A0}, $\beta \phi^{-1}(x,s) \le \phi^{-1}(y,s)$
holds for all such \(s\).
%
%
%
\end{proof}

As mentioned at the beginning of Section~\ref{sec:prelim},
the following lemma can be interpreted as  a counterpart in
our setting of the weighted power-mean inequality for
the function \(\varphi(x, t)
= t^p\) for some \(p\geq 1\).

\begin{lemma}\label{lem:wpm}
Let \(\Omega\subset \RR^n\) be open and assume that
\(\phi \in \Phi_w(\Omega)\) satisfies \eqref{aDec}.
Then, for all $\delta>0$ and \(a, b \geq 0\) and for \aev\ \(x\in\Omega \), we have
\begin{equation}
\label{wpmiPhi2}
\phi(x,a+b)\le \phi(x, (1+\delta)a)+ \frac{1}{\cd} \left(1+\frac{1}{\delta}\right)^{\phid}\phi(x,b)
\end{equation}
 and
\begin{equation}
\label{wpmiPhi}
\phi(x,a+b)\le \frac{1}{\cd} \left[(1+\delta)^{\phid}\,\phi(x,a)+
\left(1+\frac{1}{\delta}\right)^{\phid}
\phi(x,b) \right].
\end{equation}
\end{lemma}

\begin{proof}
If \(b \leq \delta a\), then the monotonicity of \(\phi(x,\cdot)\)
yields
\begin{equation*}
\begin{aligned}
\phi(x,a+b) \leq \phi(x,(1+\delta)a).
\end{aligned}
\end{equation*}
If \(a<\delta^{-1} b\), then the monotonicity of \(\phi(x,\cdot)\)
and \eqref{aDec} yield
\begin{equation*}
\begin{aligned}
\phi(x,a+b) \leq \phi(x,(1+\delta^{-1})b) \leq \frac{1}{\cd}
(1+\delta^{-1})^\phid\phi(x,b).
\end{aligned}
\end{equation*}
Thus, \eqref{wpmiPhi2} holds. Further, \eqref{wpmiPhi} follows from
\eqref{wpmiPhi2} by \eqref{aDec}.
\end{proof}

The following lemma is a variant of \cite[Lemma~4.4]{Has15} without
the assumption $\rho_{\phix}(f\chi_{\{|f|> \sigma\}})< 1$
and correspondingly weaker conclusion (see also \cite[Lemma~2.2]{HasR_pp15}).

\begin{lemma}\label{lem:key-estimate}
Let \(\Omega\subset \RR^n\) be open and assume that $\phi\in \Phi_w(\Omega)$ satisfies
Assumption~\eqref{A}.
Then, there exists $\beta'>0$ such that, for every ball $B\subset\Omega$,
$f\in L^{\phix}(\Omega)$, and \aev\ \(x\in B\), we have
\begin{equation*}
\phi\bigg(x,  \beta'  \min\bigg\{(\phi_B^-)^{-1}\Big(\frac1{|B|}\Big),
\fint_B |f(y)|\, dy\bigg\}\bigg)
\le
\fint_B \phi(y,|f(y)|) \, dy + h(x)  + \fint_B  h(y) \, dy,
\end{equation*}
where $h$ is the function provided by \eqref{A2}.
\end{lemma}

\begin{proof}
Without loss of generality, we may assume
that \(f\geq 0\). By Remarks~\ref{onPhis}
and \ref{Remark:assumptions}, we may also assume that
$\phi\in \Phi(\Omega)$. Then, \(t\mapsto \tfrac{\phi(x,t)}{t}\) is increasing
for every \(x\in\Omega\) .

Fix a ball $B\subset\Omega$, and denote $\alpha:=(\phi_B^-)^{-1}\big(\frac1{|B|}\big)$.
Let \((\sigma,\beta)\) be given by Assumption~\eqref{A}, and set
$f_1:=f\chi_{\{f>\sigma\}}$, $f_2:=f-f_1$,
and $F_i:=\fint_B f_i\, dy$ for \(i\in\{1,2\}\). Because
$\phi(x,\cdot)$ is convex and increasing,
\begin{equation}
\label{keyest1}
\phi\bigg(x, \frac{\beta}{4}\min\bigg\{\alpha, \fint_B f\, dy\bigg\}\bigg)
\le
\phi\big(x, \tfrac{\beta}2 \min\{\alpha, F_1\}\big)
+
\phi\big(x,  \beta F_2\big).
\end{equation}

We start by estimating the first term on the right-hand
side of \eqref{keyest1}.
Suppose first that $\tfrac12\min\{\alpha,  F_1\} > \sigma$. Then, by definition of $\alpha$,
$\tfrac12\min\{\alpha,  F_1\} \in [\sigma,
(\phi_B^-)^{-1}\big(\frac1{|B|}\big)]$. Thus, Lemma~\ref{onA1s}
and the monotonicity of \(\phi(x,\cdot)\) and \(\phi^-_B\) yield
$\phi\big(x, \tfrac {\beta}2 \min\{\alpha,  F_1\}\big)
\le \phi_B^-\big(\tfrac 12 F_1\big)$.
Using now  \cite[Lemma~4.3]{Has15}, we obtain
\begin{equation*}
\phi\big(x, \tfrac{\beta}2 \min\{\alpha,  F_1\}\big)
\le
\phi_B^-\big(\tfrac 12 F_1\big)
\le
\fint_B \phi_B^-(f_1(y))\, dy
\le
\fint_B \phi(y, f_1(y))\, dy.
\end{equation*}

Next, suppose that $\tfrac12\min\{\alpha,  F_1\} \le \sigma$.
By convexity and monotonicity of $\phi(x,\cdot)$, by \eqref{A0}, and by convexity again, we conclude that
\begin{equation*}
\phi\big(x, \tfrac{\beta}2 \min\{\alpha,  F_1\}\big)
\le
\phi\big(x, \beta\sigma\big) \, \frac {F_1}{2\sigma}
\le
\frac {1}{2\sigma} \fint_B f_1 (y)\,dy
\le
\frac {1}{2}\fint_B \phi(y,f_1(y))\, dy,
\end{equation*}
where, in the last inequality, we used also \eqref{A0}
together with the fact that $f_1(y)>\sigma$ in \(\{y\in\Omega\!:\, f_1(y)\not=0\} \).

To estimate the second term on the right-hand
side of \eqref{keyest1}, we invoke the convexity and
monotonicity of $\phi(x,\cdot)$
and Assumption~\eqref{A2} to obtain
\begin{equation*}
\phi\big(x,  \beta F_2\big) \le
\fint_B \phi(x, \beta f_2(y))\, dy
\leq
\fint_B \phi(y,f_2(y))\, dy + h(x) + \fint_B  h(y)\, dy.
\end{equation*}
Recalling that \(\alpha=(\phi_B^-)^{-1}\big(\frac1{|B|}\big)\) and $\phi(x,0)=0$,
the claim follows.
\end{proof}


\section{Auxiliary results}\label{sec:aux}

\begin{lemma}\label{lem:smooth}
Let $\Omega$ and $U$ be two open subsets of $\Rn$ such that $U\subset\subset\Omega$.
Let \(\phi \in \Phi_w(\Omega)\)  satisfy Assumptions~\eqref{aDec} and \eqref{loc}, and let
$(\hatrho)_\epsilon$ be a family of functions
satisfying \eqref{condition rho 1} and \eqref{condition
rho 2}. Then,
\begin{equation}
\label{forsmooth}
\begin{aligned}
\lim_{\epsilon\to0^+} \varrho^\epsilon_{\#,U}(f)
\approx  \varrho_{\phix,U}(c_n|\nabla f|)
\end{aligned}
\end{equation}
for all  $f\in C^\infty(\Omega)$, where $
c_n := \fint_{B(0,1)} |x\cdot e_1| \, dx $.
If, in addition, \(\phi(x,\cdot)\) is continuous for
all \(x\in U\), then \eqref{forsmooth} holds with equality.
\end{lemma}

\begin{remark}
Note that \(\phi(x,\cdot)\) is continuous
for
all \(x\in \Omega\) if $\phi\in \Phi(\Omega)$
satisfies
\eqref{aDec}.
\end{remark}

\begin{proof}
Let $f\in C^\infty(\Omega)$. We start by treating the case in which \(\phi(x,\cdot)\) is continuous for
all \(x\in U\). We claim that
\begin{equation}
\label{forsmooth1}
\begin{aligned}
\lim_{\epsilon\to0^+} \varrho^\epsilon_{\#,U}(f)
=  \varrho_{\phix,U}(c_n|\nabla f|).
\end{aligned}
\end{equation}

By the Taylor expansion formula, for any $x\in\Omega$
and \(y\in\Omega\) such that $|y-x|<\dist(x,\partial\Omega)$,
we have\[
f(y) = f(x) + \nabla f(x) \cdot (y-x) + R(x,y),
\]
where $R(x,y)=o(|x-y|)$ as $y\to x$.
%
%
%
Denote $ h(x,r):=\frac{2}{r}\fint_{B(x,r)}|R(x,y)|\,
dy$, for $r>0$ and $x\in\Omega_r$,
and
$ c_n = \fint_{B(0,1)} |x\cdot e_1| \, dx$.
As proved in \cite[Lemma~3.1]{HasR_pp15}, we have the point-wise estimate
\[
c_n\,r\, |\nabla f(x)| - rh(x,r)
\le
M^\#_{B(x,r)} f
\le
c_n\,r\, |\nabla f(x)| + rh(x,r).
\]
Consequently, we can define a function $\alpha:\Omega\times \mathbb{R}^+\to [-1,1]$
such that $\frac1r M^\#_{B(x,r)} f = c_n\, |\nabla f(x)| + \alpha(x,r) h(x,r)$
when $x\in \Omega_r$. Then,
\begin{equation}\label{first estimate rho epsilon}
\begin{split}
\varrho^\epsilon_{\#,U}(f)
=
 \int_0^\infty \int_{U_r} \phi(x,c_n |\nabla f(x)| +\alpha(x,r) h(x,r))  \, dx \, \hatrho(r)\, dr.
\end{split}
\end{equation}

Next, we prove that \(\limsup_{\epsilon\to 0^+} \varrho^\epsilon_{\#,U}(f)\leq \varrho_{\phi(\cdot),U}(c_n|\nabla f|) \). Fix \(\delta>0\).
Because \(c_n |\nabla f(x)| + \alpha(x,r)h(x,r) \leq
c_n |\nabla f(x)| + h(x,r) \), the  monotonicity of $\phi$  and \eqref{wpmiPhi2} yield
\begin{equation*}
\phi(x,c_n |\nabla f(x)| + \alpha(x,r)h(x,r))
\leq
 \phi(x, (1+\delta)c_n  |\nabla f(x)|) +\frac{1}{\cd} \left(1+\frac{1}{\delta}\right)^{\phid} \phi(x, h(x,r)).
\end{equation*}
Hence, invoking \eqref{first estimate rho epsilon} and \eqref{condition rho 1}, we obtain
\begin{equation}\label{upper bound limit}
\begin{aligned}
\limsup_{\epsilon\to 0^+}\varrho^\epsilon_{\#,U}(f) &
\le   \varrho_{\phi(\cdot),U}( (1+\delta)c_n |\nabla f|)\\
&\quad +  \frac{1}{\cd} \left(1+\frac{1}{\delta}\right)^{\phid}
\limsup_{\epsilon\to 0^+}
 \int_0^\infty \int_{U_r} \phi(x,  h(x,r))\, dx\,\hatrho(r)\, dr.
\end{aligned}
\end{equation}
We claim that
\begin{equation}
\label{limffih0}
\lim_{\epsilon\to 0^+}
 \int_0^\infty \int_{U_r} \phi(x,  h(x,r))\, dx\,\hatrho(r)\,
dr =0,
\end{equation}
from which the estimate on the upper limit follows
by dominated convergence as \(\delta\to 0^+\) in
\eqref{upper bound limit} taking also into account the
continuity of \(\phi(x,\cdot)\).

To prove \eqref{limffih0}, we start by observing
that because $U$ is bounded, the set $U_r$
is empty for all $r>0$ sufficiently large. Thus, there exists $r_0>0$ for which
we have
$$
\int_0^\infty\int_{U_r} \phi(x,h(x,r))\, dx\, \hatrho(r)\, dr
=
\int_0^{r_0}\int_{U_r} \phi(x,h(x,r))\, dx\, \hatrho(r)\, dr.
$$
Moreover, since $f\in C^2$, \(|R(x,y)| \leq \| f\|_{W^{2,\infty}(U)}|x-y|^2\) for all
$r\in(0,r_0),$ $x\in U_r,$ and $y\in B(x,r)$. Set
$C:=\| f\|_{W^{2,\infty}(U)}$. Then, for all $r\in(0,r_0)$ and
$x\in U_r$, \(h(x,r) \le 2Cr\). Fix \(0<\gamma < \min\{1,r_0\}\).
We have that   \(\phi(x,h(x,r)) \leq
\phi(x,2C r_0) \)
whenever \(r\in (\gamma, r_0)\); moreover, denoting by
\(c\) the monotonicity constant of  $t\mapsto\frac{\phi(x,t)}t$, we have also
\(\phi(x,h(x,r)) \leq c \gamma \phi(x,2C)\)
whenever \(r\in (0,\gamma]\). Using, in addition,
the condition \(\phi\ge 0\) and \eqref{condition rho 1}, it follows that
\begin{equation*}
\int_0^{r_0}\int_{U_r} \phi(x,h(x,r))\,
dx\, \hatrho(r)\, dr \leq c\gamma\int_U \phi(x,2C)\,dx
+   \int_U\phi(x,2C r_0)\, dx\int_\gamma^{\infty} \hatrho(r)\, dr.
\end{equation*}
In view of  \eqref{condition rho 2} and Assumption \eqref{loc}, together with
Remark~\ref{Remark:assumptions},
letting  \(\epsi\to0^+\)  in this estimate first, and then \(\gamma\to 0^+\), we obtain  \eqref{limffih0}.

Finally, we prove that \(\liminf_{\epsilon\to 0^+}
\varrho^\epsilon_{\#,U}(f)\geq
\varrho_{\phi(\cdot),U}(c_n|\nabla f|) \).  Fix \(\delta>0\), and
denote $a':=c_n |\nabla f(x)|$ and $b':=\alpha(x,r)h(x,r)$.
If   $r> 0$ and $x\in U_r$ are such that \(\alpha(x,r)
<0\), then applying  \eqref{wpmiPhi2} with $a:=\frac{a'+b'}{1+\delta}$
and
$b:=-\frac{b'}{1+\delta}$ gives
\begin{equation}
\label{liforsmooth1}
\begin{aligned}
\phi(x,c_n |\nabla f(x)|+\alpha(x,r)h(x,r))
&= \phi(x,(1+\delta)a)\geq \phi(x, a+b) -
\tfrac{1}{\cd}(1+\tfrac{1}{\delta})^\phid\phi(x,b)\\
&\geq\phi\big(x,\tfrac{c_n |\nabla f(x)|}{1+\delta}\big) -\tfrac{1}{\cd}(1+\tfrac{1}{\delta})^\phid\phi(x,h(x,r)),
\end{aligned}
\end{equation}
where we used the monotonicity of \(\phi\) and the estimate  \(0\leq -\tfrac{b'}{1+\delta}\leq  h(x,r) \). If   $r> 0$ and $x\in U_r$ are such that \(\alpha(x,r)\geq0\), then
the
 monotonicity of \(\phi\) yields
\begin{equation}
\label{liforsmooth2}
\begin{aligned}
\phi(x,c_n |\nabla f(x)|+\alpha(x,r)h(x,r))
&\geq\phi\big(x,\tfrac{c_n |\nabla f(x)|}{1+\delta}\big).
\end{aligned}
\end{equation}
Thus, by \eqref{liforsmooth1}
and
\eqref{liforsmooth2}, for every  $r> 0$ and $x\in U_r$,  we have
\begin{equation*}
\begin{aligned}
\phi(x,c_n |\nabla f(x)|+\alpha(x,r)h(x,r))
\ge
 \phi\big(x,\tfrac{c_n}{1+\delta} |\nabla f(x)|\big)
- \tfrac{1}{\cd}(1+\tfrac{1}{\delta})^\phid\phi(x,h(x,r)).
\end{aligned}
\end{equation*}
Consequently, by \eqref{first estimate rho epsilon}, it follows that
\begin{equation}\label{estbelowrhoepsi0}
\begin{aligned}
\varrho^\epsilon_{\#,U}(f)
\ge
\int_0^\infty\int_{U_r} \Big[ \phi\big(x,\tfrac{c_n}{1+\delta} |\nabla f(x)|\big)
- \tfrac{1}{\cd}(1+\tfrac{1}{\delta})^\phid\phi(x,h(x,r))\Big]\, dx\, \hatrho(r)\, dr.
\end{aligned}
\end{equation}
Fix  $0<\gamma<1$ and $r_0>0$ such that \(U_{r_0} \not=\emptyset\). By \eqref{condition rho 1} and
\eqref{condition
rho 2}, there exists \(\epsi_0=\epsi_0 (\gamma,r_0)\)
such that, for all \(\
\varepsilon\in(0,\varepsilon_0)\), we have
\begin{equation*}
\int_0^{r_0}\hatrho(r)\,dr\ge
1-\gamma.
\end{equation*}
Thus, we have also, for $0<\epsi<\epsi_0$,
\begin{align*}
\int_0^\infty \int_{U_r} \phi\big(x,\tfrac{c_n}{1+\delta} |\nabla f(x)|\big)  \, dx \, \hatrho(r)\, dr
& \ge \int_0^{r_0} \int_{U_{r_0}} \phi\big(x,\tfrac{c_n}{1+\delta} |\nabla f(x)|\big)  \, dx
\, \hatrho(r)\, dr \\
& \ge (1-\gamma)\int_{U_{r_0}}
\phi\big(x,\tfrac{c_n}{1+\delta}  |\nabla f(x)|\big)  \, dx.
\end{align*}
This estimate,  \eqref{estbelowrhoepsi0}, and
\eqref{limffih0} yield
\begin{equation*}
\liminf_{\epsilon\to 0^+}\varrho^\epsilon_{\#,U}(f)
\ge {(1-\gamma)}
  \varrho_{\phix,U_{r_0}}\big(\tfrac{c_n}{1+\delta} |\nabla f|\big).
\end{equation*}
The conclusion then follows
by using the
continuity of \(\phi(x,\cdot)\) and by letting $\delta\to 0^+$, $\gamma\to 0^+$, and
$r_0\to 0^+$ in this order. This completes the proof
of \eqref{forsmooth1} under the continuity assumption.

Suppose now that $\phi$ is a general weak $\Phi$-function. By Remarks~\ref{onPhis}
and \ref{Remark:assumptions},
there exists $\psi\in \Phi(\Omega)$
satisfying the same assumptions as \(\phi\) and such that \(\psi\approx\phi\).
Then, there is \(C>0\)  such that \(\tfrac1C \phi\leq \psi \leq C \phi \). Hence, by the first part of the
proof,
\begin{align*}
&\frac{1}{C^2}\varrho_{\phix,U}(c_n|\nabla f|) \leq \frac{1}{C}\varrho_{\psix,U}(c_n|\nabla f|)
=\frac{1}{C}\liminf_{\epsilon\to0^+} \varrho^\epsilon_{\#,U, \psi}( f)\leq\liminf_{\epsilon\to0^+} \varrho^\epsilon_{\#,U, \phi}( f)
 \\
&\quad\le
\limsup_{\epsilon\to0^+}
\varrho^\epsilon_{\#,U, \phi}( f) \leq C  \limsup_{\epsilon\to0^+}
\varrho^\epsilon_{\#,U, \psi}( f)= C  \varrho_{\psix,U}( c_n|\nabla f|)
\le  C^2 \varrho_{\phix,U}(c_n|\nabla f|).\qedhere
\end{align*}
\end{proof}

Next, we derive an auxiliary upper bound which holds for all functions $f\in L^1_\loc(\Omega)$
with $|\nabla f| \in L^\phix (\Omega)$.

\begin{lemma}\label{lem:sobolev}
Let $\Omega$ be an open set of $\Rn$, let
\(\phi \in \Phi_w(\Omega)\) satisfy Assumptions~\eqref{A}
and
\eqref{aDec}, and let
$(\hatrho)_\epsilon$ be a family of functions
satisfying \eqref{condition rho 1} and \eqref{condition
rho 2}.
If $f\in L^1_\loc(\Omega)$ and $|\nabla f| \in L^\phix(\Omega)$, then
\[
\varrho^\epsilon_{\#,\Omega}(f) \le
c\, \max\big\{ \|\nabla f\|^\phiu_{\phix,\Omega}, \|\nabla f\|^\phid_{\phix,\Omega}\big\}.
\]
\end{lemma}

\begin{proof}
By Remarks~\ref{onPhis}, \ref{casePhi}, and \ref{Remark:assumptions}
we may assume  without loss of generality that \(\phi\in
\Phi(\Omega)\). Then, in view of  Assumptions~\eqref{A0}
and \eqref{aInc}$_1$, we have $L^\phix(B)\subset L^1(B)$
for every bounded set \(B\subset\Omega\) by \cite[Lemma~4.4]{HarHK_pp16}.

Let $f\in L^1_\loc(\Omega)$
with $|\nabla f| \in L^\phix(\Omega)$ be such
that
$\|\nabla f\|_{\phix,\Omega} \le 1$. Fix $r>0$ and $x\in\Omega_r$.
By the Poincar\'e inequality in $L^1$, we have
\begin{equation}
\label{byPoinc}
\begin{aligned}
\frac1r M^\#_{B(x,r)}f = \frac1r\fint_{B(x,r)} | f(y) - f_{B(x,r)}
| \, dy
\le
c  \fint_{B(x,r)} | \nabla f(y) | \, dy.
\end{aligned}
\end{equation}
On the other hand, in view of Lemma~\ref{onA1s}, we may invoke \cite[Lemma 4.4]{Has16} (with \(\gamma=1)\) that gives the existence of a constant, \(\beta'\in(0,1)\), depending only on the constants in \eqref{A},  such that
\begin{equation}\label{byH44}
\begin{aligned}
\phi\bigg(x,\beta'\fint_{B(x,r)}
|\nabla f(y)|\, dy\bigg)  \leq \fint_{B(x,r)}
\phi(y,|\nabla f(y)|)\, dy+ h(x)+ \fint_{B(x,r)}h(y)\,dy.
\end{aligned}
\end{equation}
Using the monotonicity of \(\phi\), \eqref{aDec},  \eqref{byPoinc}, and \eqref{byH44},
we obtain
\begin{equation}\label{byPoincandH44}
\begin{aligned}
\varrho^\epsilon_{\#,\Omega}(f)
&\lesssim\int_0^\infty \int_{\Omega_r} \phi\bigg(x,\beta'\fint_{B(x,r)}
|\nabla f(y)|\, dy\bigg)dx \, \hatrho(r)\,
dr\\
&\leq
 \int_0^\infty \int_{\Omega_r} \bigg(\fint_{B(x,r)}
\phi(y,|\nabla f(y)|)\, dy + h(x)+ \fint_{B(x,r)}h(y)\,dy\bigg)
\, dx \, \hatrho(r)\, dr.
\end{aligned}
\end{equation}

Next, by changing the order of integration, we observe that
\begin{align*}
\int_{\Omega_r} \fint_{B(x,r)} \phi(y,|\nabla
f(y)|) \, dy \, dx
 &= \int_{\Omega_r} \int_\Omega \phi(y,|\nabla
f(y)|) \frac{\chi_{B(y,r)}(x)}{|B(x,r)|} \,
dy\,dx
\\
& \le \int_\Omega \phi(y,|\nabla f(y)|) \, dy
=\varrho_{\phix,\Omega}(|\nabla f|)\le 1.
\end{align*}
Similarly,
\begin{equation*}
\begin{aligned}
\int_{\Omega_r} \fint_{B(x,r)} h(y) \, dy \, dx \leq
\int_\Omega h(y)\, dy.
\end{aligned}
\end{equation*}
Consequently, because $h\in L^1(\Omega)$
and $\int_0^\infty \hatrho(r) \, dr =1$, from \eqref{byPoincandH44}
we
conclude that
\begin{equation}\label{estimatecasele 1}
\varrho^\epsilon_{\#,\Omega}(f) \le c\ \text{
for all }f\in L^1_\loc(\Omega)\text{ with }|\nabla
f| \in L^\phix(\Omega) \text{ and } \|\nabla f\|_{\phix,\Omega}
\le 1.
\end{equation}

By considering the cases $\lambda\le 1$ and $\lambda>1$ and use
Assumptions \eqref{aInc}$_1$ and \eqref{aDec}, respectively, 
we find that $\varrho^\epsilon_{\#,\Omega}(f) \le
c\varrho^\epsilon_{\#,\Omega}(\frac f \lambda) \max\{ \lambda^\phiu, \lambda^\phid\}$
for all $\lambda>0$ and $f\in L^1_\loc(\Omega)$ with $|\nabla f| \in L^\phix(\Omega)$.
Using this estimate with $\lambda:= \|\nabla f\|_{\phix,\Omega}+\delta$ for $\delta>0$ and then
invoking \eqref{estimatecasele 1}, it follows that
\begin{align*}
\varrho^\epsilon_{\#,\Omega}(f) & \le
c\varrho^\epsilon_{\#,\Omega}\left(\frac f {\|\nabla f\|_{\phix,\Omega}+\delta}\right)
\max\big\{ {(\|\nabla f\|_{\phix,\Omega}+\delta)}^\phiu,
{(\|\nabla f\|_{\phix,\Omega}+\delta)}^{\phid}\big\}\\
&\le c \max\big\{ {(\|\nabla f\|_{\phix,\Omega}+\delta)}^\phiu,
{(\|\nabla f\|_{\phix,\Omega}+\delta)}^{\phid}\big\}.
\end{align*}
Letting  $\delta\to
0^+$, we  conclude the proof of Lemma~\ref{lem:sobolev}.
\end{proof}


\section{Main results}\label{sect:main}

In this section we prove our main result, which
provides a characterization of generalized Orlicz spaces.
Theorem~\ref{main theorem} is an immediate
consequence of Propositions~\ref{prop:SobolevCase}, \ref{prop:SobolevCase2} and
\ref{prop:showsSobolev} below.

\begin{proposition}\label{prop:SobolevCase}
Let $\Omega\subset\Rn$ be open, let $\phi\in \Phi_w(\Omega)$
satisfy Assumptions~\eqref{A}, \eqref{aInc}, and \eqref{aDec}, and let
$(\hatrho)_\epsilon$ be a family of functions
satisfying \eqref{condition rho 1} and \eqref{condition
rho 2}.
Assume that $f\in L^1_\loc(\Omega)$
with $|\nabla f| \in L^\phix(\Omega)$. Then,
\begin{equation}\label{main limit general}
\lim_{\epsilon\to 0^+} \varrho^\epsilon_{\#,\Omega}(f)
\approx \varrho_{\phix,\Omega}(c_n |\nabla f|)
< \infty
\end{equation}
where $ c_n = \fint_{B(0,1)} |x\cdot e_1| \, dx$.
If, in addition, \(\phi(x,\cdot)\) is continuous for
all \(x\in \Omega\), then \eqref{main limit general} holds with equality.
\end{proposition}


\begin{proof}
We start with the upper bound.
Let $U\subset\subset\Omega$. Using the same arguments
as in \cite[Theorems 6.5 and 6.6]{HarHK_pp16}, we find
$g_\nu\in C^\infty(\Omega)$, $\nu\in\N$, such that
$(\nabla g_\nu)_{n\in\NN}$ converges to $\nabla f$ in $L^\phix(U;\mathbb{R}^n)$. For  $r>0$ and $x\in U_r$, we have
$M^\#_{B(x,r)}f \le M^\#_{B(x,r)}(f-g_\nu)+ M^\#_{B(x,r)}g_\nu$
by the triangle inequality. Combining this estimate with \eqref{wpmiPhi2} applied to $a=\tfrac1r M^\#_{B(x,r)}g_\nu$ and $b=\tfrac1rM^\#_{B(x,r)}(f-g_\nu)$, we obtain
\begin{equation*}
\varrho^\epsilon_{\#,U}(f)
\le \varrho^\epsilon_{\#,U}((1+\delta)g_\nu) +
 c_\delta\varrho^\epsilon_{\#,U}(f-g_\nu)
\end{equation*}
for all \(\delta>0\), where we also used the monotonicity
of \(\phi(x,\cdot)\).
Invoking now
Lemmas~\ref{lem:smooth} and \ref{lem:sobolev}, we conclude
that
\begin{equation}\label{almlimsup1}
\limsup_{\epsilon\to 0^+} \varrho^\epsilon_{\#,U}(f)
\le  \varrho_{\phix,U}(c_n (1+\delta)|\nabla g_\nu|)
+ c_\delta\!
\max\big\{ \|\nabla (f-g_\nu)\|_{\phix,U}^{\phiu},
\|\nabla (f-g_\nu)\|_{\phix,U}^{\phid}\big\}.
\end{equation}

Because \(\|\nabla (f-g_\nu)\|_{\phix,U}\to0\), it follows from \eqref{aDec} that
\(\lim_{\nu\to\infty}\varrho_{\phix, U}(c_n (1+\delta)|\nabla (f-g_\nu)|)=0\).
Thus, letting \(\nu\to\infty\) in \eqref{almlimsup1} and using \eqref{wpmiPhi2}
 again, we obtain
\begin{equation}\label{onubmod}
\limsup_{\epsilon\to 0^+} \varrho^\epsilon_{\#,U}(f)
\le
\varrho_{\phix, U}(c_n (1+\delta)^2|\nabla f|)
\le
\varrho_{\phix, \Omega}(c_n (1+\delta)^2|\nabla f|).
\end{equation}
By \eqref{aDec}, $\varphi(x,c_n(1+\delta)^2|\nabla f(x)|) \le c\varphi(x,c_n|\nabla f(x)|)$.
This proves that \(\limsup_{\epsilon\to 0^+} \varrho^\epsilon_{\#,U}(f)
\le c
\varrho_{\phix, U}(c_n |\nabla f|)\)  for a general $\phi$. If, in addition,
 $\phi(x,\cdot)$ is  continuous,  we use \(c\varphi(x,c_n|\nabla
f(x)|)\) as a majorant and let $\delta\to
0$ in \eqref{onubmod}. Then, Lebesgue's dominated convergence theorem yields
\[
\limsup_{\epsilon\to 0^+} \varrho^\epsilon_{\#,U}(f)
\le \varrho_{\phix,\Omega}(c_n |\nabla f|).
\]

We claim that
\begin{equation}
\label{almlimsup3}
\begin{aligned}
\limsup_{U \nearrow \Omega} \limsup_{\epsi\to0^+} \int_0^\infty\int_{\Omega_r\setminus U_r}
\phi\big(x,\tfrac1r M^\#_{B(x,r)}f\big) \,dx \,  \hatrho(r)\,
dr = 0,
\end{aligned}
\end{equation}
which, together with the bound on $\limsup_{\epsilon\to 0^+} \varrho^\epsilon_{\#,U}(f)$
established above, concludes the proof of the upper bound.

To prove \eqref{almlimsup3}, we observe that
replacing $f$ by $\tilde f := \Vert\nabla
f\Vert_{\phix,\Omega}^{-1}f$
if necessary, we may assume that $\|\nabla
f\|_{\phix,\Omega}\le 1$.
Indeed, by \eqref{aDec}, if \eqref{almlimsup3} holds
for \(\tilde f\), then it also holds
for $f$.
Arguing as in
Lemma~\ref{lem:sobolev}, we find that
\begin{equation}
\label{almlimsup4}
\begin{aligned}
& \int_0^\infty\int_{\Omega_r\setminus U_r}
\phi\big(x, \tfrac1r M^\#_{B(x,r)}f \big) \,
dx\,  \hatrho(r)\, dr \\
& \quad \le c \int_0^\infty
\int_{\Omega_r\setminus U_r} \bigg(\fint_{B(x,r)}
\phi(y,|\nabla f(y)|)\, dy  +h(x)+ \fint_{B(x,r)}h(y)\,dy\bigg) \, dx
\,  \hatrho(r)\, dr.
\end{aligned}
\end{equation}
Let $\gamma>0$. If $r\le \frac\gamma2$, then
\(y\in \Omega\backslash U_\gamma\)
 whenever \(x\in \Omega_r\backslash U_r \). Therefore,
  we may estimate the right-hand side of \eqref{almlimsup4}
  as follows:
\begin{equation}
\label{almlimsup5}
\begin{aligned}
& \int_0^\infty
\int_{\Omega_r\setminus U_r} \bigg(\fint_{B(x,r)}
\phi(y,|\nabla f(y)|)\, dy  +h(x)+ \fint_{B(x,r)}h(y)\,dy\bigg) \, dx
\,  \hatrho(r)\, dr \\
&\quad \le
\int_0^\infty
\int_{\Omega_r\setminus U_r} \bigg(\int_{\Omega\backslash U_\gamma} (\phi(y,|\nabla f(y)|)+h(y))\tfrac{\chi_{B(y,r)}(x)}{|B(0,r)|}\, dy + h(x)\bigg) \, dx
\,  \hatrho(r)\, dr \\
&\qquad +
\int_\frac{\gamma}{2}^\infty
\int_{\Omega_r} \bigg(\fint_{B(x,r)} \phi(y,|\nabla f(y)|)\, dy  +h(x)+ \fint_{B(x,r)}h(y)\,dy\bigg)
\, dx\,  \hatrho(r)\, dr.
\end{aligned}
\end{equation}
Changing the order of integration as in Lemma~\ref{lem:sobolev}, from \eqref{almlimsup4} and \eqref{almlimsup5}, we obtain
\begin{equation*}
\begin{aligned}
&\int_0^\infty\int_{\Omega_r\setminus U_r}
\phi\big(x, \tfrac1r M^\#_{B(x,r)}f \big) \,
dx\,  \hatrho(r)\, dr\\
&\quad \lesssim \varrho_{\phix, \Omega\backslash U_\gamma}(|\nabla f|)+2\|h\|_{L^1(\Omega\backslash U_\gamma)} + \big(\varrho_{\phix,\Omega}(|\nabla f|) + 2\|h\|_{L^1(\Omega)}\big)
\int_\frac\gamma2^\infty
\hatrho(r)\, dr.
\end{aligned}
\end{equation*}
Then, \eqref{almlimsup3} follows by \eqref{condition rho 2} as $\epsilon\to 0$,
$\gamma\to 0$, and \(U\to\Omega\), in
this order.

\smallskip
It remains to prove the lower bound.
The proof of this estimate is similar to the proof of the upper bound once we show that for all \(U\subset\subset\Omega\),
we have
\begin{equation}\label{ftlb}
\varrho^\epsilon_{\#,U}(f)
\ge \varrho^\epsilon_{\#,U}(\tfrac1{(1+\delta)}g_\nu) -  c_\delta\varrho^\epsilon_{\#,U}(f-g_\nu),
\end{equation}
where \(c_\delta = \tfrac1\cd \big(1 + \tfrac1\delta\big)^\phid\).
Fix \(r>0\) and \(x\in U_r\),
and set $a':=\tfrac1r M^\#_{B(x,r)}g_\nu$
and $b':=\tfrac1rM^\#_{B(x,r)}(f-g_\nu)$. Note that \(\tfrac1r M^\#_{B(x,r)}f \geq |a'-b'|\).
We consider two cases:
\begin{itemize}
\item
If \(a'-b' \geq 0\), then using \eqref{wpmiPhi2} with \(a = \tfrac{a'-b'}{1+\delta}\) and  \(b = \tfrac{b'}{1+\delta}\),
we obtain
\begin{equation*}
\begin{aligned}
\phi(x, a'-b') \geq \phi\big(x, \tfrac{a'}{1+\delta}\big)
- c_\delta \phi\big(x, \tfrac{b'}{1+\delta}\big)\geq \phi\big(x, \tfrac{a'}{1+\delta}\big)
- c_\delta \phi(x, b'),
\end{aligned}
\end{equation*}
where in the last inequality we used the fact  that \(\phi(x,\cdot)\)
is increasing and \(\tfrac1{1+\delta}\leq1\).
\item
If \(a'-b' \leq 0\), then
\begin{equation*}
\begin{aligned}
\phi\big(x, \tfrac{a'}{1+\delta}\big) \leq \phi\big(x, \tfrac{b'}{1+\delta}\big) \leq
 c_\delta \phi(x, b')
\end{aligned}
\end{equation*}
because \(\phi(x,\cdot)\)
is increasing,  \(\tfrac1{1+\delta}\leq1\), and \(c_\delta
\geq 1\). Thus,
\begin{equation*}
\begin{aligned}
 0\geq\phi\big(x, \tfrac{a'}{1+\delta}\big) -c_\delta \phi(x, b').
\end{aligned}
\end{equation*}
\end{itemize}
Splitting the inner integral defining \(\varrho^\epsilon_{\#,U}(f)\) according to these two cases, we conclude that   \eqref{ftlb} holds.
\end{proof}

\begin{remark}
In the proof of the previous proposition the assumption
\eqref{aInc} was used only for the density of smooth functions
in the Sobolev space. Presumably, \eqref{aInc} is not really needed
for this, but it was used in the cited reference.
\end{remark}

\begin{proposition}\label{prop:SobolevCase2}
Under the assumptions of Proposition~\ref{prop:SobolevCase},  we have
\begin{equation}
\label{eq:limnormse}
\begin{aligned}
\lim_{\epsilon\to 0^+}\| f\|^\epsilon_{\#,\Omega} \approx
c_n \| \nabla f\|_{\phix,\Omega}.
\end{aligned}
\end{equation}
If, in addition, \(\phi\in \Phi(\Omega)\), then \eqref{eq:limnormse} holds with equality.
\end{proposition}

\begin{proof}
Suppose first that \(\phi\in \Phi(\Omega)\). Since $\phi$ is also doubling,
\(\phi(x,\cdot)\)  is continuous; thus,
 the modular inequality \eqref{main limit general} holds with equality.
Fixing \(\delta>0\) and applying  this equality to the function
\begin{equation*}
\begin{aligned}
g_\delta:= \frac {f}{c_n(1+\delta)( \|\nabla f\|_{\phix,\Omega}+\delta)},
\end{aligned}
\end{equation*}
we obtain
\begin{equation*}
\begin{aligned}
\lim_{\epsilon\to 0^+} \varrho^\epsilon_{\#,\Omega}(g_\delta)
=\varrho_{\phix,\Omega}(c_n |\nabla g_\delta|) \leq  \tfrac{1}{1+\delta}\int_\Omega
\phi\Big(x, \frac{|\nabla f|}{ \|\nabla f\|_{\phix,\Omega}+\delta}\Big)\, dx <1,
\end{aligned}
\end{equation*}
where we used convexity in the last estimate. Thus,
\(\varrho^\epsilon_{\#,\Omega}(g_\delta ) \leq 1\) for all sufficiently small \(\epsi>0\). Consequently, also $\|g_\delta\|^{\epsilon}_{\#,\Omega}\le
1$ for all sufficiently small
\(\epsi>0\); that is,  $\|f\|^{\epsilon}_{\#,\Omega}\le
c_n(1+\delta) ( \|\nabla f\|_{\phix,\Omega}+\delta) $. Letting $\epsi\to 0$ first and then \(\delta\to0\), we obtain \(\limsup_{\epsilon\to 0^+}\| f\|^\epsilon_{\#,\Omega}
\leq
c_n \| \nabla f\|_{\phix,\Omega}\).

Next, we  prove the opposite inequality.
Fix \(\delta>0\), let \(\epsi_j \to0\) as \(j\to\infty\)
be such that \(\liminf_{\epsilon\to 0^+} \|f\|^\epsilon_{\#,\Omega}
= \lim_{j\to\infty}  \|f\|^{\epsilon_j}_{\#,\Omega}\), and set
\begin{equation*}
\begin{aligned}
g_\delta:= \frac {f}{\lim_{j\to\infty}  \|f\|^{\epsilon_j}_{\#,\Omega}+\delta}.
\end{aligned}
\end{equation*}
Because
\(\lim_{j\to\infty}  \|g_\delta\|^{\epsilon_j}_{\#,\Omega}
< 1\), we conclude that \(\|g_\delta\|^{\epsilon_j}_{\#,\Omega}\le 1\)
for all sufficiently large \(j\in\NN\). For all such  \(j\in\NN\),
\(\varrho^{\epsilon_j}_{\#,\Omega}(g_\delta)\le 1\) by
the definition of the norm. Letting \(j\to\infty\),
the previous proposition yields
 $\varrho_{\phix,\Omega}(c_n |\nabla g_\delta|)\le 1$;  so,
$\|c_n |\nabla g_\delta|\|_{\phix,\Omega}\le 1$. Therefore,
$\|c_n|\nabla f|\|_{\phix,\Omega}\le
\lim_{j\to\infty}  \|f\|^{\epsilon_j}_{\#,\Omega}+\delta
= \liminf_{\epsilon\to 0^+} \|f\|^\epsilon_{\#,\Omega}+\delta$.
Letting
\(\delta\to0\), we obtain the desired
inequality.
\smallskip

This completes the proof in the case $\phi\in \Phi(\Omega)$.
If $\phi\in \Phi_w(\Omega)$, we find $\psi\in \Phi(\Omega)$
with $\phi\simeq\psi$ (Remarks~\ref{onPhis} and \ref{Remark:assumptions}).
Then, $\| \nabla f\|_\phix \approx \|\nabla f\|_\psix$
and similarly for the $\#$-norm; so, the claim follows from the first part.
\end{proof}

The next lemma shows that the condition
$\limsup_{\epsilon\to 0^+} \varrho^\epsilon_{\#,\Omega}(f) < \infty$ implies that $f$ is
locally in a Sobolev space. The estimate for the
norm obtained in this way is not uniform, however. Nevertheless, this
information is used later to prove a uniform estimate.

\begin{lemma}\label{lem:pMinus}
Let $\Omega\subset\Rn$ be open, let $\phi\in \Phi_w(\Omega)$
satisfy Assumptions~\eqref{A}, \eqref{aInc}, and \eqref{aDec},
and let $(\hatrho)_\epsilon$ be a family of functions
satisfying \eqref{condition rho 1} and \eqref{condition rho 2}.
Assume that $f\in L^1_{\rm loc}(\Omega)$ and
\[
\limsup_{\epsilon\to 0^+} \varrho^\epsilon_{\#,\Omega}(f)
< \infty.
\]
Then, there is a constant $c>0$ such that, for every $U\subset\subset\Omega$, $\int_U \phi_U^-(|\nabla f(x)|)\, dx \le c$.
Moreover, $\nabla f\in L^{\phiu}_\loc(\Omega;\mathbb{R}^n)$.
\end{lemma}

\begin{proof}
Let $U\subset\subset\Omega$ and note that $\phi_U^-\in \Phi_w$ (the only
nontrivial condition is the limit at
infinity, which follows from \eqref{A0}).
By \cite[Lemma~2.2]{HarHK_pp16}, there exists $\xi\in \Phi$ with $\xi \simeq \phi_U^-$ for which the equivalent constant
depends only on the monotonicity constant
of \(t\mapsto \tfrac{\phi(x,t)}{t}\).
Then, by \eqref{aDec}, we can find \(c>0\),
independent of \(U\), such that \(c^{-1}\phi_U^-
\leq \xi \leq c\, \phi_U^-\) .

Fix \(\delta>0\), and let \(G_\delta\) be a standard
mollifier; that is, $G_\delta(x)=\delta^{-n}G(x/\delta)$, where
$G\in C_0^\infty(B(0,1))$ is a non-negative function,
radially symmetric, and  $\int_{B(0,1)}G(x)\,dx=1$. Let
 $\tau>0$ be such that $U\subset\subset\Omega_\tau$. Then, for $0<\delta<\tau$,  $G_\delta*f \in C^\infty(\Omega_\tau)$
and, by Lemma~\ref{lem:smooth} with
\(\phi=\xi\), \(U:=U_\delta\), and \(\Omega:=\Omega_\tau\),
\begin{equation}\label{applying lem:smooth}
\begin{aligned}
\lim_{\epsilon \to 0^+}\int_0^\infty
\int_{U_{\delta+r}} \xi\Big(\tfrac1r M^\#_{B(x,r)}(G_\delta*f) \Big) \, dx
\, \hatrho(r)\, dr
&=\varrho_{\xi,U_\delta}(c_n|\nabla(G_\delta*f)|)
\\
&\ge \varrho_{\xi,U_\tau}(c_n|\nabla(G_\delta*f)|).
\end{aligned}
\end{equation}

On the other hand, by the triangle inequality and a change on the order of
integration, we obtain, for all $r>0$ and  $x\in U_r$,
\begin{equation}\label{fornextproof}
\begin{aligned}
M^\#_{B(x,r)}(G_\delta*f)
& = \fint_{B(x,r)} | (G_\delta*f)(y) - (G_\delta*f)_{B(x,r)}
| \, dy \\
& = \fint_{B(x,r)} | (G_\delta*f)(y) - (G_\delta*f_{B(\cdot,r)})(x)
| \, dy
\le (G_\delta*M^\#_{B(\cdot,r)}f)(x).
\end{aligned}
\end{equation}
Hence, the monotonicity of \(\xi\) yields
\begin{equation*}
\begin{aligned}
\xi\big(\tfrac1r M^\#_{B(x,r)}(G_\delta*f)
\big) \leq  \xi\big((G_\delta*\tfrac1r M^\#_{B(\cdot,r)}f)(x)
\big).
\end{aligned}
\end{equation*}
For $x\in\Omega_r$, define $g_r(x):= \frac1r M_{B(x,r)}^\#f$.
Since $G_\delta\, dx$ is a probability measure and $\xi$ is convex,
\begin{equation*}
\begin{aligned}
\xi\left(\tfrac1r M^\#_{B(x,r)}(G_\delta*f)\right)
\le
\xi((G_\delta* g_r)(x))
\le (G_\delta * \xi(g_r))(x).
\end{aligned}
\end{equation*}

Integrating $G_\delta * \xi(g_r)$ over
$U_{r+\delta}$ and changing variables and the order of integration,
we obtain
\begin{equation*}
\begin{aligned}
\int_{U_{\delta+r}} \int_{B(x,\delta)} G_\delta(x-z) \xi(g_r(z))\, dz \, dx
&=
\int_{B(0,\delta)} G_\delta(w)\int_{U_{\delta+r}}  \,    \xi(g_r(x-w))\,dx dw\\
&\le
\int_{U_r}  \xi(g_r(y))\, dy.
\end{aligned}
\end{equation*}
Combining the previous two estimates, we conclude that
\begin{equation*}
\begin{aligned}
\int_{U_{\delta+r}} \xi\Big(\tfrac1r M^\#_{B(x,r)}(G_\delta*f)
\Big)\, dx
\le
\int_{U_r} \xi(g_r(x)) \, dx
\le c
\int_{U_r} \phi\Big(x,\tfrac1r M^\#_{B(x,r)}f \Big)\, dx,
\end{aligned}
\end{equation*}
where in the last inequality we used the definition of $g_r$ and the estimate $\xi(t)\le c\, \phi(x,t)$
for $x\in U$. Consequently,
\begin{equation}
\label{eq:estMs3}
\begin{aligned}
\limsup_{\epsilon \to 0^+}\int_0^\infty
\int_{U_{\delta+r}} \xi\Big(\tfrac1r M^\#_{B(x,r)}(G_\delta*f) \Big) \, dx
\, \hatrho(r)\, dr
& \le c \limsup_{\epsilon \to 0^+} \varrho_{\#,U}^\epsilon(f) \le c \limsup_{\epsilon \to 0^+} \varrho_{\#,\Omega}^\epsilon(f)  < \infty,
\end{aligned}
\end{equation}
which, together with \eqref{applying lem:smooth}, shows that
$(\nabla (G_\delta * f))_\delta$ is bounded in $L^{\xi}(U_\tau;\Rn)$.

By \eqref{A0}, \eqref{aInc}, and \eqref{aDec}, $L^{\xi}(U_\tau;\Rn)$
is reflexive (see \cite[Lemma~2.4 and Proposition~4.6]{HarHK_pp16}).
Since $(\nabla (G_\delta * f))_\delta$ is a bounded sequence, it has a
subsequence which converges weakly in $L^{\xi}(U_\tau;\Rn)$
to a function $l\in L^{\xi}(U_\tau;\Rn)$.
Using the fact that $G_\delta*f$ converges to $f$ in $L^1(U)$,
it follows from the definition of weak derivative that $l=\nabla f$.
By weak semi-continuity \cite[Theorem~2.2.8]{DieHHR11}, \eqref{applying lem:smooth}, and \eqref{eq:estMs3}, we obtain that
\begin{equation*}
\begin{aligned}
\int_{U_\tau} \phi^-_U(c_n|\nabla f(x)|)\,dx
&\leq c \rho_{\xi,U_\tau}(c_n|\nabla f|)
\leq c \liminf_{\delta \to 0^+} \rho_{\xi,U_\tau}(c_n|\nabla (G_\delta*f)|)
\\
&\le
c\limsup_{\epsilon \to 0^+} \, \varrho_{\#,\Omega}^\epsilon(f)<\infty.
\end{aligned}
\end{equation*}
Letting \(\tau\to 0\), we obtain $\int_U \phi_U^-(|\nabla f(x)|)\, dx \le c$ by
 monotone convergence and \eqref{aDec}.

Finally, we observe that Assumptions~\eqref{aInc} and \eqref{A0} imply
that $t^{\phiu} \le c(\phi_U^-(t) + 1)$ with $c>0$ independent of $U$.
Hence,
\[
\int_U |\nabla f(x)|^{\phiu} \, dx
\le
c\int_U \phi_U^-(|\nabla f(x)|) + 1\, dx
\le
c + c\, |U|,
\]
which yields \(\nabla f\in L^{\phiu}_\loc(\Omega;\mathbb{R}^n).\)
\end{proof}

We now remove the local-condition from the previous lemma:

\begin{proposition}\label{prop:showsSobolev}
Let $\Omega\subset\Rn$ be open, let $\phi\in \Phi_w(\Omega)$
satisfy Assumptions~\eqref{A}, \eqref{aInc}, and \eqref{aDec},
and let  $(\hatrho)_\epsilon$ be a family of functions
satisfying \eqref{condition rho 1} and \eqref{condition rho 2}.
Assume that $f\in L^1_{\rm loc}(\Omega)$ and
\[
\limsup_{\epsilon\to 0^+} \varrho^\epsilon_{\#,\Omega}(f)
< \infty.
\]
Then, $\nabla f\in L^{\phix}(\Omega;\mathbb{R}^n)$.
\end{proposition}

\begin{proof}
Let $U\subset\subset V\subset\subset\Omega$ and $\delta\in (0,\frac16)$ be such that
$U^\delta\subset V\subset\Omega_\delta$. Let
$G_\delta$ and $g_r$ be as in the previous lemma.
By Remarks~\ref{onPhis} and \ref{Remark:assumptions}, there exists $\xi\in \Phi(\Omega)$ equivalent to $\phi$ which
satisfies the same assumptions.
In Lemma~\ref{lem:pMinus}, we proved that
$\int_U \xi_U^-(|\nabla f|)\, dx \le c$ for every $U\subset\subset\Omega$.
In view of \eqref{aDec}, by scaling the function \(f\), if necessary, we may assume that
\begin{equation}
\label{eq:byscalingf}
\begin{aligned}
\int_B \xi_B^-(|\nabla f(x)|)\, dx \le 3^{-n} \text{ for every ball
$B\subset\subset\Omega$
with $|B|\le 1$.}
\end{aligned}
\end{equation}

Let $r \in (0, \delta]$ and $x\in U_r$. By the Poincar\'e and triangle inequalities,
\begin{equation*}
\begin{aligned}
\frac1r M^\#_{B(x,r)}(G_\delta*f)
\le
c \fint_{B(x,r)} |\nabla( G_\delta*f)(y)|\, dy
\le
c \fint_{B(x,r)} \int_{B(y,\delta)} G_\delta(y-z)
|\nabla f(z)|\, dz\, dy.
\end{aligned}
\end{equation*}
Note that
$G_\delta \le c \frac{\chi_{B(0,\delta)}}{|B(0,\delta)|}$;
hence, by \cite[Lemma~4.3]{Has15} and \eqref{eq:byscalingf},
together with the previous estimate, we obtain
\begin{equation*}
\begin{aligned}
\tfrac1r M^\#_{B(x,r)}(G_\delta*f)
&\le
c \fint_{B(x,r)} \fint_{B(y,\delta)} |\nabla f(z)|\,
dz\, dy \\
&\le
c \fint_{B(x,r)} (\xi_{B(y,\delta)} ^-)^{-1}\bigg(
\fint_{B(y,\delta)}  \xi_{B(y,\delta)} ^-(|\nabla f(z)|)\,
dz\bigg)\, dy \\
&\le c \fint_{B(x,r)} (\xi_{B(y,\delta)} ^-)^{-1}\Big(
\tfrac{1}{3^n|B(y,\delta)|}\Big)\, dy.
\end{aligned}
\end{equation*}
Observe that $B(y,\delta)\subset B(x,3\delta)$
for every $y\in B(x,r)$, and
 $|B(y,\delta)|=|B(x,\delta)|$ and $|B(x,3\delta)|=3^n |B(x,\delta)|$.
By Lemmas~\ref{Lem:auxformain} and \ref{lem:As},
\begin{equation*}
\begin{aligned}
(\xi_{B(y,\delta)} ^-)^{-1}(\tfrac{1}{3^n|B(y,\delta)|}) \approx
(\xi_{B(x,3\delta)} ^-)^{-1}(\tfrac{1}{3^n|B(x,\delta)|}) \approx
(\xi_{B(x,\delta)} ^-)^{-1}(\tfrac{1}{3^n|B(x,\delta)|})
\end{aligned}
\end{equation*}
when $y\in B(x,r)$.
Then, because  \((\xi_{B(x,\delta)} ^-)^{-1 }\) is  increasing, we obtain
\[
\tfrac1r M^\#_{B(x,r)}(G_\delta*f) \le c
 (\xi_{B(x,\delta)} ^-)^{-1}(3^{-n}|{B(x,\delta)} |^{-1})\leq c
 (\xi_{B(x,\delta)} ^-)^{-1}(|{B(x,\delta)} |^{-1}).
\]

On the other hand, by \eqref{fornextproof}
, we  have also
$\tfrac1r M^\#_{B(x,r)}(G_\delta*f) \le (G_\delta*g_r)(x)\leq
c \fint_{B(x,\delta)} |g_r|\, dy$.
Thus,
\[
\frac1r M^\#_{B(x,r)}(G_\delta*f)
\le
c \min \bigg\{ \fint_{B(x,\delta)} |g_r|\, dy , (\xi_{B(x,\delta)}^-)^{-1}\big(\tfrac1{|{B(x,\delta)} |}\big) \bigg\}.
\]
Then,  by the monotonicity of \(\xi(x,\cdot)\)
and  Lemma~\ref{lem:key-estimate}, invoking
\eqref{aDec} if necessary,
it follows  that
\[
\xi\Big(x, \tfrac1r M^\#_{B(x,r)}(G_\delta*f) \Big)
\le
c \bigg(\fint_{B(x,\delta)} \big(\xi(y,g_r(y)) + h(y)\big)\, dy + h(x)\bigg).
\]
Integrating this estimate over $x\in U_r$, changing the order of integration  as in Lemma~\ref{lem:sobolev}, and
using the inclusion $(U_r)^\delta \subset V_r$, we obtain
\begin{equation*}
\begin{aligned}
\int_{U_r} \xi\big(x,\tfrac1r M^\#_{B(x,r)}(G_\delta*f)
\big)\, dx &\le
c\bigg( \int_{(U_r)^\delta} \xi(y,g_r(y)) + h(y)\, dy
+ \int_{U_r} h(x)\, dx\bigg)\\
& \le c\bigg( \int_{V_r} \xi\big(y,\tfrac1r M^\#_{B(y,r)}f\big) \, dy+2\Vert  h\Vert_{L^1(\Omega)}\bigg)
\end {aligned}
\end{equation*}
for \(r\le \delta\).

Next, we consider $r \in (\delta,\infty)$.
Fix \(x\in U_r\), and recall that \(\nabla
f \in L^{\phiu}(V;\RR^n)\) by Lemma~\ref{lem:pMinus}. By the Poincar\'e and H\"older inequalities,
we have
\begin{equation*}
\begin{aligned}
\tfrac1r M^\#_{B(x,r)}(G_\delta*f) &\leq
c\fint_{B(x,r)} |\nabla(G_\delta*f)(y)|\,dy
\leq c \fint_{B(x,r)} \Vert G_\delta\Vert_{(\phiu)',
B(0,\delta)} \Vert \nabla f\Vert_{\phiu,
B(y,\delta)}\,dy\\
&\leq c_\delta \Vert \nabla f\Vert_{\phiu,
U^\delta} \leq c_\delta \Vert \nabla f\Vert_{\phiu,
V}.
\end{aligned}
\end{equation*}
Hence,
\begin{equation*}
\int_{U_r} \xi\Big(x, \tfrac1r M^\#_{B(x,r)}(G_\delta*f)\Big)\,dx
\leq \int_{U_r} \xi\big(x,c_\delta \Vert \nabla f\Vert_{\phiu, V} \big)\,dx
\end{equation*}
for $r> \delta$.

Combining the two cases, \(r\leq \delta\)
and \(r>\delta\), we obtain
\begin{align*}
&\int_0^\infty \int_{U_r} \xi\big(x,\tfrac1r
M^\#_{B(x,r)}(G_\delta*f) \big)\, dx
\, \hatrho(r)\, dr \\
& \quad \le
c \int_0^\delta \bigg(\int_{V_r} \xi\big(x,\tfrac1r
M^\#_{B(x,r)}f\big) \, dx+2\Vert  h\Vert_{L^1(\Omega)}
\bigg) \hatrho(r)\, dr
+ \int_\delta^\infty \int_{U_r}
\xi\big(x,c_\delta \Vert
\nabla f\Vert_{\phiu, V} \big)\,dx\, \hatrho(r)\,
dr.
\end{align*}
The second term on the right-hand side
of the previous inequality tends to zero
as $\epsilon\to 0$ by \eqref{loc} and
\eqref{condition
rho 2}.  Consequently,
\begin{align*}
\limsup_{\epsilon \to 0}\int_0^\infty
\int_{U_r} \xi\Big(x,\tfrac1r M^\#_{B(x,r)}(G_\delta*f)
\Big)\, dx
\, \hatrho(r)\, dr
\le
c \limsup_{\epsilon \to 0} \rho_{\#,V}^\epsilon(f) + c < \infty
\end{align*}
by hypothesis, where we also used \eqref{condition rho 1}  and
the fact that \(\xi\approx \phi\).

Thus,  arguing as in  Lemma~\ref{lem:pMinus},
we conclude that $\nabla
f\in L^{\xi(\cdot)}(U;\mathbb{R}^n)$.
Hence, also $\nabla
f\in L^{\phi(\cdot)}(U;\mathbb{R}^n)$
(see Remark~\ref{casePhi}).

Finally, we appeal to Proposition~\ref{prop:SobolevCase}
to conclude that
\[
\varrho_{\phix,U}(c_n |\nabla f|)
\leq c
\limsup_{\epsilon\to 0^+} \varrho^\epsilon_{\#,U}(f)
\le c
\limsup_{\epsilon\to 0^+} \varrho^\epsilon_{\#,\Omega}(f)
< \infty.
\]
Because the upper bound in the last estimate is independent of $U$,
we conclude the proof of Proposition~\ref{prop:showsSobolev} by monotone convergence
as $U\nearrow \Omega$.
\end{proof}

Combining the propositions of this section, we arrive at the following
theorem.

\begin{theorem}\label{main theorem gen}
Let $\Omega\subset\Rn$ be an open set, let $\phi\in \Phi_w(\Omega)$  satisfy Assumptions~\eqref{A}, \eqref{aInc}, and \eqref{aDec}, and let $(\hatrho)_\epsilon$ 
be a family of functions
satisfying \eqref{condition rho 1} and \eqref{condition rho 2}.
Assume  that $f\in L^1_{\rm loc}(\Omega)$.
Then,
\begin{equation*}
\begin{aligned}
\nabla f\in L^{\phix}(\Omega;\Rn)
\quad\Leftrightarrow\quad
\limsup_{\epsilon\to 0^+} \varrho^\epsilon_{\#,\Omega}(f) <\infty.
\end{aligned}
\end{equation*}
In this case,
\begin{equation*}
\lim_{\epsilon\to0^+} \varrho^\epsilon_{\#,\Omega}(f)\approx \varrho_{\phix,\Omega}(c_n |\nabla f|)
\qquad\text{and}\qquad
\lim_{\epsilon\to0^+}\| f\|^\epsilon_{\#,\Omega} \approx c_n \| \nabla f\|_{\phix,\Omega},
\end{equation*}
where $c_n := \fint_{B(0,1)} |x\cdot e_1| \, dx$.
\end{theorem}


\section*{Acknowledgements}
This work was partially supported by the Funda\c{c}\~ao para a Ci\^encia e a
Tecnologia (Portuguese Foundation for Science and Technology) through the projects
UID/MAT/00297/ 2013 (Centro de Matemática e Aplicações) and EXPL/MAT-CAL/0840/2013.
Part of this work was done while the authors enjoyed the hospitality of University of Turku and
of the Centro de Matem\'atica e Aplica\c{c}\~{o}es (CMA), Faculdade de Ci\^{e}ncias e Tecnologia, Universidade Nova de Lisboa.


\end{document}